\documentclass[preprint,11pt]{elsarticle}

\usepackage{amssymb}
\usepackage{amsmath}
 \usepackage{amsthm}
\usepackage{subfigure} 
\usepackage{graphicx}
\usepackage{pinlabel}

\usepackage{fullpage}

\newtheorem{theorem}{Theorem}
\newtheorem{proposition}[theorem]{Proposition}
\newtheorem{lemma}[theorem]{Lemma}

\newtheorem{corollary}[theorem]{Corollary}

\theoremstyle{definition}
\newtheorem{definition}[theorem]{Definition}

\newtheorem{conjecture}[theorem]{Conjecture}
\newtheorem{example}[theorem]{Example}

\theoremstyle{remark}
\newtheorem{remark}[theorem]{Remark}

\newcommand{\ba}{\setminus}
\newcommand{\F}{\mathcal{F}}
\newcommand{\btu}{\bigtriangleup}

\journal{European Journal of Combinatorics}

\begin{document}

\begin{frontmatter}

\author[rhul]{Iain Moffatt}
    \ead{iain.moffatt@rhul.ac.uk}
    \author[wits]{Eunice Mphako-Banda}
    \ead{eunice.mphako-banda@wits.ac.za}
\title{Handle slides for delta-matroids}

 \address[rhul]{Department of Mathematics, Royal Holloway, University of London, Egham, Surrey, TW20 0EX, United Kingdom.}
 \address[wits]{School of Mathematics, University of the Witwatersrand,Wits 2050, Johannesburg, South Africa}

\begin{abstract}
A classic exercise in the topology of surfaces is to show that, using handle slides, every disc-band surface, or 1-vertex ribbon graph, can be put in a canonical form consisting of the connected sum of orientable loops, and either  non-orientable loops or pairs of interlaced orientable loops. Motivated by the principle that ribbon graph theory informs delta-matroid theory, we find the delta-matroid analogue of this surface classification. We show that, using a delta-matroid analogue of handle slides,  every binary delta-matroid in which the empty set is feasible can be written in a canonical form consisting of the direct sum of the delta-matroids of orientable loops, and either  non-orientable loops or pairs of interlaced orientable loops. Our delta-matroid results are compatible with the surface results in the sense that they are their ribbon graphic delta-matroidal analogues. 
\end{abstract}

\begin{keyword}
delta-matroid \sep disc-band surface \sep handle slide \sep ribbon graph
\MSC[2010] 05B35 \sep 05C10
\end{keyword}
\end{frontmatter}

\section{Overview and background}\label{s.1}
Matroid theory is often thought of as a generalisation of graph theory.   W.~Tutte famously observed  that, ``If a theorem about graphs can be expressed in terms of edges and circuits alone it probably exemplifies a more general theorem about matroids'' (see~\cite{Oxley01}). The merit of this point of view is that the more `tactile' area of graph theory can serve as a guide for matroid theory, in the sense that  results and properties for graphs can indicate what results and properties about matroids might hold.  In~\cite{CMNR1} and \cite{CMNR2}, C.~Chun et al. proposed that a similar relationship holds between topological graph theory and delta-matroid theory, writing ``If a theorem about embedded graphs can be expressed in terms of its spanning quasi-trees then it probably exemplifies a more general theorem about delta-matroids''. Taking advantage of this principle, here  we use classical results from surface topology to guide us to a classification of binary delta-matroids.

Informally, a ribbon graph is a ``topological graph'', whose vertices are discs and  edges are ribbons,  that arises from a regular neighbourhood of a graph in a surface. Formally, a \emph{ribbon graph} $G =\left(  V, E  \right)$ consists of a set of discs $V$  whose elements are  \emph{vertices}, a set of discs $E$ whose elements are \emph{edges}, and is such that (i) the vertices and edges intersect in disjoint line segments; (ii) each such line segment lies on the boundary of precisely one vertex and precisely one edge; and (iii) every edge contains exactly two such line segments. 
We note  that ribbon graphs describe exactly  cellularly embedded graphs, and  refer the reader to \cite{EMMbook}, or \cite{GT87} where they are called reduced band decompositions, for further background on ribbon graphs. A ribbon graph is \emph{non-orientable} if it contains a ribbon subgraph that is homeomorphic to a M\"obius band, and is \emph{orientable} otherwise.
A ribbon graph with exactly one vertex is called a \emph{bouquet}. 
An edge $e$ of a ribbon graph is a \emph{loop} if it is incident with exactly one vertex. A loop is \emph{non-orientable} if together with its incident vertex it forms a M\"obius band, and is \emph{orientable} otherwise.  Two loops $e$ and $f$ are \emph{interlaced} if they are met in the cyclic order $efef$ when travelling round the boundary of a vertex.
We let $B_{i,j,k}$ denote the bouquet shown in Figure~\ref{f5c} consisting of $i$ orientable loops, $j$ pairs of interlaced orientable loops, and $k$ non-orientable loops. 

A \emph{handle slide} is the move on ribbon graphs defined in Figures~\ref{f5a} and~\ref{f5b} which `slides' the end of one edge over an  edge adjacent to it in the cyclic order at a vertex. (We make no assumptions about the order that the points $1, \ldots , 6$ in the figure appear on a vertex.) A standard exercise in low-dimensional topology is to show that every bouquet  can be put into the canonical form $B_{i,j,k}$ using handle slides  (see for example, \cite{cart,crom,grif}, and note that in topology bouquets are often called disc-band surfaces).
 In fact, we can always assume that in the canonical form $B_{i,j,k}$,  one of $j$ or $k$ is zero. The following records the results of this exercise.
\begin{proposition} \label{hs} 
For each bouquet $B$ and for some $i,j,k$, there is a sequence of handle slides taking $B$ to $B_{i,j,0}$ if  $B$ is orientable,  or $B_{i,0,k}$, with $k\neq 0$, if $B$ is non-orientable. Furthermore, if some sequences of handle slides take $B$ to   $B_{i,j,k}$ and to $B_{p,q,r}$ then $i=p$, and so $B$ is taken to a unique form $B_{i,j,0}$ or $B_{i,0,k}$ by handle slides.
\end{proposition}
This result is essentially the classification surfaces with boundary up to homeomorphism restricted to bouquets: $j$ is the number of tori making up the surface, $k$ the number of real projective planes, and $i+1$ is the number of holes in the surface.

\begin{figure}
\centering
\subfigure[Slide $a$ over $b$ to the right.]{
\labellist
 \small\hair 2pt
\pinlabel {$1$}  [r] at  10 17 
\pinlabel {$2$} [l]  at   55 17 
\pinlabel {$3$} [r] at    81 17 
\pinlabel {$4$}  [l] at   126 17 
\pinlabel {$5$}  [r] at    154  17 
\pinlabel {$6$}  [l] at   199 17 
\pinlabel {$a$}   at  68 41
\pinlabel {$b$}   at  144 41 

\endlabellist
\includegraphics[scale=0.8]{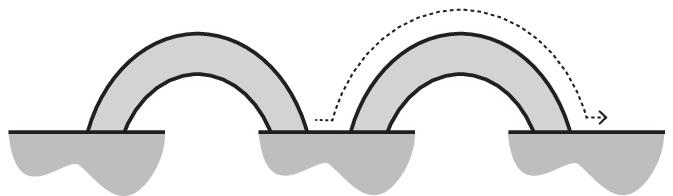}
\label{f5a}
}
\hspace{5mm}
\subfigure[Slide $a$ over $b$ to the left.]{
\labellist
 \small\hair 2pt
\pinlabel {$1$}  [r] at  10 17 
\pinlabel {$2$} [l]  at   55 17 
\pinlabel {$3$} [r] at    81 17 
\pinlabel {$4$}  [l] at   126 17 
\pinlabel {$5$}  [r] at    154  17 
\pinlabel {$6$}  [l] at   199 17 
\pinlabel {$a$}   at  110 73
\pinlabel {$b$}   at  140 41 
\endlabellist
\includegraphics[scale=.8]{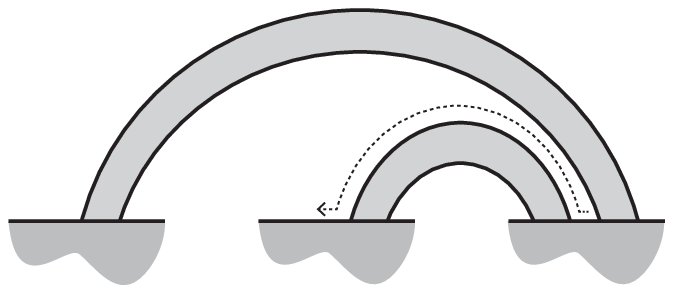}
\label{f5b}
}

\subfigure[The bouquet $B_{i,j,k}$.]{
\labellist
 \small\hair 2pt
\pinlabel {$\overset{i}{\overbrace{\quad\quad\quad\quad}}$}   at  35 67  
\pinlabel {$\overset{j}{\overbrace{\quad\quad\quad\quad\quad\quad}}$}   at  102 70 
\pinlabel {$\overset{k}{\overbrace{\quad\quad\quad\quad}}$}   at  165 67
\endlabellist
\includegraphics[scale=1]{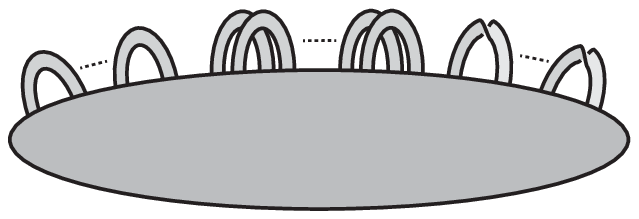}
\label{f5c}
}

\caption{Handle slides.}
\label{f5}
\end{figure}

Following  the principle of \cite{CMNR1} that ribbon graphs serve as a guide for delta-matroids, we look for the delta-matroid analogue of Proposition~\ref{hs}. Our aim is to find a classification of delta-matroids up to ``homeomorphism'' that is consistent with this surface result.

Delta-matroids, introduced by A.~Bouchet in \cite{ab1}, generalise matroids.  Recall the \emph{symmetric difference}, $X\triangle Y$,  of sets $X$ and $Y$  is $(X\cup Y)\backslash (X\cap Y)$.
A \emph{set system} is a pair $(E,\mathcal{F})$  consisting of a finite set $E$ and a collection $\mathcal{F}$ of subsets of $E$.
A \emph{delta-matroid} $D$ is a set system $(E,\mathcal{F})$ in which  $\mathcal{F}$ is non-empty and  satisfies the \emph{Symmetric Exchange Axiom}:
for all $X,Y\in \mathcal{F}$, if there is an element $u\in X\triangle Y$, then there is an element $v\in X\triangle Y$ such that $X\triangle \{u,v\}\in \mathcal{F}$. 
 Elements of $\mathcal{F}$ are called \emph{feasible sets} and $E$ is the \emph{ground set}.  We often use $\mathcal{F}(D)$ and $E(D)$ to denote the set of feasible sets and the ground set, respectively, of $D$.
 If its feasible sets are all of the same parity, $D$ is  \emph{even}, otherwise it is \emph{odd}.  
  It is not hard to see that if we impose the extra condition that the feasible sets are equicardinal, the definition of a delta-matroid becomes a reformulation of the bases definition of a matroid.
 Thus a matroid is exactly a delta-matroid whose feasible sets are all of the same size (in which case the feasible sets are the bases of the matroid). 
 
 If  $D=(E,\mathcal{F})$ and $D'=(E',\mathcal{F}')$ are delta-matroids with $E\cap E'=\emptyset$, the \emph{direct sum}, $D\oplus D'$,  of $D$ and  $D'$ is the delta-matroid with ground set 
$E\cup E'$ and feasible sets $\{F\cup F'\mid F\in \mathcal{F}\text{ and } F'\in\mathcal{F}'\}$.  We define $D_{i,j,k}$ to be the delta-matroid arising as the direct sum of $i$ copies of $(\{e\}, \{\emptyset\})$, $j$ copies of $(\{e,f\}, \{\emptyset, \{e,f\}\})$, and $k$ copies of $(\{e\}, \{\emptyset, \{e\}\})$. (Strictly speaking we sum isomorphic copies of these delta-matroids having mutually disjoint ground sets.)

Here we prove the analogue of Proposition~\ref{hs} for binary delta-matroids. 
The terms handle slide and binary delta-matroid in the theorem statement are defined in Sections~\ref{s.2} and~\ref{s.3}, respectively.
\begin{theorem}\label{t.1}
Let $D=(E,\mathcal{F})$ be a binary delta-matroid in which the empty set is feasible. Then, for some  $i,j,k$, there is a sequence of handle slides taking $D$ to $D_{i,j,0}$ if  $D$ is even,  or $D_{i,0,k}$, with $k\neq 0$, if $D$ is odd. Furthermore, if some sequences of handle slides take $D$ to   $D_{i,j,k}$ and to $D_{p,q,r}$ then $i=p$, and so $D$ is taken to a unique form $D_{i,j,0}$ or $D_{i,0,k}$ by handle slides.
\end{theorem}
This theorem is the analogue of Proposition~\ref{hs} in the following sense.  Every ribbon graph gives rise to a delta-matroid, as described in Section~\ref{s.2}. If we replace each ribbon graph term in Proposition~\ref{hs} with its delta-matroid analogue, a bouquet becomes a delta-matroid in which the empty set is feasible, $D_{i,j,k}$ is the delta-matroid of $B_{i,j,k}$, we define a delta-matroid handle slide in  Section~\ref{s.2} as the analogue of a handle slide on a bouquet, being orientable becomes being even, and non-orientable becomes odd.  Thus Theorem~\ref{t.1} gives a classification of a class of  delta-matroids up to ``homeomorphism'', showing  how the interplay between  ribbon graphs and delta-matroids can be exploited to obtain structural results about delta-matroids.

Although it is an analogue, it is important to note that Theorem~\ref{t.1}  is \emph{not} a generalisation of Proposition~\ref{hs} since the latter can not be recovered from the former. (This is since, using terminology we shortly introduce, a handle slide of $a$ over $b$ may be defined for a ribbon graphic delta-matroid but not for the corresponding edges of a ribbon graph, see Remark~\ref{r.1}.)

\section{Defining handle slides for delta-matroids}\label{s.2}
In this section we determine the analogue of a handle slide for delta-matroids. We start by recalling  how a delta-matroid can be associated with a ribbon graph.  A \emph{quasi-tree} is a ribbon graph  with exactly one boundary component.  A ribbon graph $H$ is a \emph{spanning ribbon subgraph} of a ribbon graph $G=(V,E)$  if $H$ can be obtained from $G$ by deleting some of its edges (in particular, this means $V(H)=V(G)$). Abusing notation slightly, we say that a spanning ribbon subgraph $Q$ of $G$ is a \emph{spanning quasi-tree} of $G$ if $Q$ restricts to a spanning quasi-tree of each connected component of $G$.  The \emph{delta-matroid of $G$}, denoted $D(G)$, is $(E(G), \F(G))$ where $E(G)$ is the edge set of $G$ and 
\[   \F(G) =\{  F \subseteq E(G)  \mid   F \text{ is the edge set of a spanning quasi-tree of }G    \}. \]
It follows by results of Bouchet from  \cite{ab2} that $D(G)$ is a delta-matroid. (Bouchet worked in the language of transition systems and medial graphs. The framework  used here is from \cite{CMNR1}.) A delta-matroid is \emph{ribbon graphic} if it is isomorphic to the delta-matroid of a ribbon graph.
\begin{example}\label{examp1}
If $G$ is a plane graph then the spanning quasi-trees of $G$ are exactly the maximal spanning forests of $G$. Since the latter form the collection of  bases for the cycle matroid $M(G)$ of $G$ we have that for plane graphs  $D(G)=M(G)$.  Delta-matroids can therefore be viewed as the  analogue of matroids for topological graph theory (see \cite{CMNR1,CMNR2}, where this point of view was proposed,  for further  discussion on this).  A consequence of this is  that, for any ribbon graph $G$, the empty set is feasible in   $D(G)$ if and only if $G$ is a disjoint union of bouquets.
\end{example}
\begin{example}\label{examp2}
For the ribbon graphs $B_{i,j,k}$ defined in Section~\ref{s.1} and illustrated in Figure~\ref{f5c}, we have  $D(B_{i,j,k})=D_{i,j,k}$, where $D_{i,j,k}$ is also as in Section~\ref{s.1}.  
\end{example}

\begin{definition}\label{d1}
Let $D=(E,\F)$ be a set system, and $a,b\in E$ with $a\neq b$. We define $D_{ab}$ to be the set system  $(E,\F_{ab})$ where 
\[ \F_{ab}:= \F \btu \{ X\cup a \mid X\cup b \in \F \text{ and } X\subseteq E\ba \{a,b\}   \}. \]
We say that there is \emph{a sequence of handle slides taking   $D$  to  $D'$}  if $D'=(\cdots ( (D_{a_1b_1})_{a_2b_2}) \cdots )_{a_nb_n}$ for some $a_1,b_1,\ldots, a_n,b_n \in E$, and we call the move taking $D$ to $D_{ab}$ a \emph{handle slide} taking $a$ over $b$.
\end{definition}
Note that $(D_{ab})_{ab}=D$ and that  handle slides define an equivalence relation on set systems. 

\begin{example}\label{examp3} 
If $D=(E,\F)$ with  $E=\{1,2,3\}$ and $\F=\{ \{1,2,3\}, \{1,2\}, \{1,3\}, \{2,3\},\emptyset\}$, then $\F_{12}=\{ \{1,2,3\}, \{1,2\}, \{2,3\},\emptyset\}$.
\end{example}

The following theorem shows that Definition~\ref{d1} provides the delta-matroid analogue of a handle slide.
\begin{theorem}\label{t.hs}
Let $G=(V,E)$ be a ribbon graph,  $a$ and $b$ be distinct edges of $G$ with neighbouring ends, and $G_{ab}$ be the ribbon graph obtained from $G$ by handle sliding $a$ over $b$ as in Figure~\ref{f5a} to~\ref{f5b}. Then 
\[D(G_{ab})=D(G)_{ab}.\]
\end{theorem}
\begin{proof}
Handle slides act disjointly on direct sums of delta-matroids and on connected components of ribbon graphs. Furthermore, the delta-matroid of a disconnected ribbon graph is the direct sum of the delta-matroids of its connected components. This means that, without loss of generality, we can assume that $G$ is connected.

Every feasible set in $D(G_{ab})$ and $D(G)_{ab}$ is of the form $X$, $X\cup a$, $X\cup b$ or $X\cup\{a,b\}$ for some $X\subseteq E\ba \{a,b\}$. 
Suppose $1,\ldots, 6$ are the points on the boundary components of $G$ and $G_{ab}$ shown in Figures~\ref{f5a} and~\ref{f5b}. Each   $X\subseteq E\ba \{a,b\}$ defines spanning ribbon subgraphs of $G$ and  of $G_{ab}$. The boundary components of the spanning ribbon  subgraphs $(V,X)$ connect the points $1,\ldots, 6$ in some way. 
For each  $X\subseteq E\ba \{a,b\}$ such that at least one of $X$, $X\cup a$, $X\cup b$ or $X\cup\{a,b\}$ is feasible (i.e., defines a spanning quasi-tree), Table~\ref{table1} shows all of the ways that the points  $1,\ldots, 6$ can be connected to each other in the boundary components of the corresponding ribbon subgraphs, and whether $X$, $X\cup a$, $X\cup b$ and $X\cup\{a,b\}$ is feasible in $D(G_{ab})$ or $D(G)_{ab}$.  For example, the entry  $(13)(24)(56)$ indicates that there are arcs (13), (24), and (56) in the boundary components of the spanning ribbon subgraphs defined by $X$. In this case, assuming at least one of $X$, $X\cup a$, $X\cup b$ or $X\cup\{a,b\}$ is feasible, it must be that $X\cup b$ and $X\cup\{a,b\}$ are feasible in $D(G)$; and $X\cup a$, $X\cup b$, and $X\cup\{a,b\}$ are feasible in $D(G_{ab})$ (as all other sets will have too many boundary components). It is then readily seen from the table that $\F(G_{ab})=\F(G)_{ab}$, as required.
\end{proof}

\begin{table}
\begin{center}
\begin{tabular}{|c|c|c|}
\hline
Connection in $(V,X)$ & $\F(G)$ & $\F(G_{ab})$ \\ \hline\hline
(12)(34)(56) & $X\cup\{a,b\}$  & $X\cup\{a,b\}$ \\ \hline
(12)(35)(46) & $X\cup a$, $X\cup\{a,b\}$ & $X\cup a$, $X\cup\{a,b\}$\\ \hline
(12)(36)(45) & $X\cup a$ & $X\cup a$\\ \hline
(13)(24)(56) & $X\cup b$, $X\cup\{a,b\}$ & $X\cup a$, $X\cup b$, $X\cup\{a,b\}$\\ \hline
(13)(25)(46) & $X$, $X\cup a$, $X\cup b$, $X\cup\{a,b\}$ & $X$, $X\cup b$, $X\cup\{a,b\}$\\ \hline
(13)(26)(45) & $X$, $X\cup a$  & $X$, $X\cup a$ \\ \hline
(14)(23)(56) & $X\cup b$ & $X\cup a$, $X\cup b$ \\ \hline
(14)(25)(36) & $X$, $X\cup\{a,b\}$& $X$, $X\cup\{a,b\}$ \\ \hline
(14)(26)(35) & $X$, $X\cup b$, $X\cup\{a,b\}$ & $X$, $X\cup a$, $X\cup b$, $X\cup\{a,b\}$\\ \hline
(15)(23)(46) & $X$, $X\cup b$ & $X$, $X\cup a$, $X\cup b$ \\ \hline
(15)(24)(36) & $X$, $X\cup a$, $X\cup\{a,b\}$ & $X$, $X\cup a$, $X\cup\{a,b\}$ \\ \hline
(15)(26)(34) &  $X\cup a$, $X\cup b$, $X\cup\{a,b\}$ &  $X\cup b$, $X\cup\{a,b\}$\\ \hline
(16)(23)(45) & $X$ & $X$\\ \hline
(16)(24)(35) & $X$, $X\cup a$, $X\cup b$ & $X$, $X\cup b$\\ \hline
(16)(25)(34) & $X\cup a$, $X\cup b$ &  $X\cup b$ \\ \hline
\end{tabular}
\end{center}
\caption{A case analysis for the proof of Theorem~\ref{t.hs}.}
\label{table1}
\end{table}

\begin{remark}\label{r.1}
A key difference between handle slides of ribbon graphs and of delta-matroids is that in a ribbon graph $G_{ab}$ can be formed only if $a$ and $b$ have adjacent ends, whereas in a delta-matroid $D_{ab}$ can be formed, without restriction, for all $a,b\in E$. (A consequence of this is that Theorem~\ref{t.hs} does not show that the set of ribbon graphic delta-matroids is closed under handle slides.)
Since $G_{ab}$ can be formed with respect to only certain edges $a$ and $b$, it is natural to ask if there is a corresponding concept of  ``allowed handle slides'' in a delta-matroid. The answer is no.
 To see why consider  the orientable bouquet $B$ with cyclic  order of edges around its vertex $1a12a23b34b4$, and the orientable  bouquet $B'$ with cyclic  order of edges around its vertex $21a12ab43b34$. Then a handle slide taking $a$ over $b$ is not possible in $B$ but is possible in $B'$. However $D(B)=D(B')$. Thus you cannot tell the ``allowed'' handle slides of a ribbon graph  from its delta-matroid alone.  
\end{remark}

\begin{remark}\label{r.2}
 Proposition~\ref{hs} and Theorem~\ref{t.hs}  immediately give a version of Theorem~\ref{t.1} for ribbon graphic delta-matroids. However, this version of the theorem is much weaker than might at first be expected. If $D$ is ribbon graphic then $D=D(G)$ for some ribbon graph $G$. Applying Proposition~\ref{t.hs} to $G$ then taking the delta-matroid of each ribbon graph will give a proof of the first part of Theorem~\ref{t.1} (that $D$ can be put in the form $D_{i,j,0}$  or $D_{i,0,k}$ depending on parity) for ribbon graphic delta-matroids. However, the uniqueness results from the second part of Theorem~\ref{t.1} do not follow in this way. This is because there may be sequences of handle slides that take you outside of the class of ribbon graphic delta-matroids (c.f. Remark~\ref{r.1}). However, we will see later that the uniqueness part of the result does indeed hold  for ribbon graphic delta-matroids (see Corollary~\ref{c.1}).
\end{remark}

We defined handle slides in terms of set systems. It is natural to ask if the set of delta-matroids is closed under handle slides.
 Example~\ref{examp3} shows that in general this is not the case: although $D$ is a delta-matroid, $D_{ab}$ is not.
 The delta-matroid from  Example~\ref{examp3} is one of A.~Bouchet and A.~Duchamp's excluded minors for binary delta-matroids from \cite{BD91}. We are thus led  to the question of whether the set of binary delta-matroids is closed under handle slides, and we turn our attention to this.

\section{Binary delta-matroids and the proof of Theorem~\ref{t.1}}\label{s.3}
Let $\mathbb{K}$ be a  field. For a finite set $E$, let $M$ be a skew-symmetric $|E|\times |E|$ matrix over $\mathbb{K}$ with rows and columns indexed by the elements of $E$. In all of our matrices, $e\in E$ indexes the $i$-th row if and only if it indexes the $i$-th column. 
Let $M\left[ A\right]$ be the principal submatrix of $M$ induced by the set $A\subseteq E$. By convention $M[\emptyset]$ is considered to be non-singular. 
Bouchet  showed in~\cite{abrep} that a delta-matroid $D(M)$ can be obtained by taking $E$ to be the ground set and $A\subseteq E$ to be feasible if and only if  $M[A]$ is non-singular over $\mathbb{K}$. 

The  \emph{twist} of a delta-matroid $D=(E,{\mathcal{F}})$ with respect to $A\subseteq E$, is the delta-matroid  $D* A:=(E,\{A\bigtriangleup X \mid  X\in \mathcal{F}\})$. It was shown by Bouchet in~\cite{ab1} that   $D* A$ is indeed  a delta-matroid.  
A delta-matroid is  \emph{representable over $\mathbb{K}$} if it has a twist that is isomorphic to $D(M)$ for some skew-symmetric matrix $M$ over $\mathbb{K}$. A delta-matroid representable over $GF(2)$ is called \emph{binary}. 
We note that ribbon graphic delta-matroids are binary (see \cite{abrep}), and also record the following result.
\begin{lemma}[Bouchet  \cite{abrep}]\label{l.rep}
Let $E$ be a finite set, $A\subseteq E$, and $M$ be a skew-symmetric $|E|\times |E|$ matrix over a field $\mathbb{K}$ with rows and columns indexed by $E$. Then if   $D=D(M)$ and $\emptyset \in \F(D\ast A)$, we have  $D\ast A=D(N)$ for some  skew-symmetric $|E|\times |E|$ matrix $N$ over  $\mathbb{K}$.
\end{lemma}

We now describe handle slides in terms of matrices.
\begin{definition}\label{d.2}
Let $E$ be a finite set and $M$ be a symmetric $|E|\times|E|$ matrix over $GF(2)$ with rows and columns indexed by the elements of $E$, and let $a,b\in E$ with $a\neq b$. We define  $M_{ab}$  to be the matrix obtained from $M$ by 
adding the column of $b$ to the column of $a$, then, in the resulting matrix, adding the row of $b$ to the row of $a$. 
We say that $M_{ab}$ is obtained by a \emph{handle slide}, or by \emph{handle sliding $a$ over $b$}.
\end{definition}
Note that in Definition~\ref{d.2} adding the  row of $b$ to the row of $a$ then, in the resulting matrix, the column of $b$ to the column of $a$ also results in the matrix $M_{ab}$.  Definition~\ref{d.2} by no means describes a new operation on matrices. For example the operation was considered by R.~Kirby in the context of handle slides and 3-manifolds in \cite{MR0467753}.

The following theorem shows that all the concepts of handle slides defined here agree.
\begin{theorem}\label{con2}
Let $M$  be a symmetric matrix over $GF(2)$. Then \[   D(M_{ab})=D(M)_{ab}. \]
\end{theorem}
\begin{proof}
We need to show that $D(M_{ab})$ and $D(M)_{ab}$ have the same feasible sets. In view of the definition of handle slides, Definition~\ref{d1}, it suffices to prove that, for $Y\subseteq E$, 
\begin{enumerate}
\item \label{con2.a}  $ \det\left(  M_{ab}[Y] \right) = \det\left(  M[Y] \right)  $ if $a\notin Y$,
\item \label{con2.b} $ \det\left(  M_{ab}[Y] \right) = \det\left(  M[Y] \right)  $ if $a,b\in Y$, and
\item \label{con2.c} $\det(M_{ab}[Y])=\det(M[Y])+\det(M[Y \btu \{a, b\}])$ if $a\in Y$ and $b\notin Y$.
\end{enumerate}

The first  item is trivial since $M[Y]=M_{ab}[Y]$ when $a\notin Y$.

For the second item, suppose that $a,b\in Y$. Observe that in this case applying the construction in Definition~\ref{d.2} to   $M[Y]$ results in $M_{ab}[Y]$  (i.e., $(M[Y])_{ab}=M_{ab}[Y]$). Since adding one row or column of a matrix to another row or column does not change the determinant,  $\det(M[Y])=\det(M_{ab}[Y])$.

For the third item, suppose that $a\in Y$ and $b\notin Y$. Set $X=Y\ba a$, so that $Y=X\cup a$ and $Y \btu \{a, b\} = X\cup b$. We then need to show \begin{equation}\label{e.con2a}
\det(M_{ab}[X\cup a])=\det(M[X\cup a])+\det(M[X\cup b]).
\end{equation}
 (We will work in terms of the set $X$, rather than $Y$, as it simplifies the exposition.)  
Suppose that $M[X\cup \{a,b\}]=[a_{i,j}]_{1\leq i,j\leq n}$. Without loss of generality, we  assume that $a$ indexes the first row and column of  $M[X\cup \{a,b\}]$, and $b$ indexes the second.   
Then $M[X]  =   [a_{i,j}]_{3\leq i,j\leq n} $;  
$M[X\cup a]$ is the $(n-1)\times(n-1)$ matrix obtained from   $M[X\cup \{a,b\}]$ by deleting  its second row and column;  
$M[X\cup b]$ is the $(n-1)\times(n-1)$ matrix obtained from   $M[X\cup \{a,b\}]$ by deleting  its first row and column; 
and $M_{ab}[X\cup a]$ is  the $(n-1)\times(n-1)$ matrix whose first row is $  \begin{bmatrix} a_{1,1}+a_{2,2} & a_{1,3}+a_{2,3} & \cdots & a_{1,n}+a_{2,n}\end{bmatrix}$, first column is   $\begin{bmatrix} a_{1,1}+a_{2,2} & a_{3,1}+a_{3,2} & \cdots & a_{n,1}+a_{n,2}\end{bmatrix}^T$, with the rest of the matrix given by  $M[X]$.

Letting $M[X]_{i,j}$ denote the matrix obtained by deleting the $i$-th row and $j$-th column of $M[X]$, using the Laplace (cofactor) expansion of the determinant, expanding down the first row and column,  gives
\begin{multline}\label{e.con2b}
\det(M_{ab}[X\cup a]) = \left( a_{1,1}\det (M[X])+  \sum_{3\leq i,j\leq n}   a_{1,i}a_{j,1}\det (M[X]_{i,j})  \right)
\\  + \left( a_{2,2}\det (M[X])+  \sum_{3\leq i,j\leq n}   a_{2,i}a_{j,2}\det (M[X]_{i,j})  \right)
\\ +  \left( \sum_{3\leq i,j\leq n}   a_{1,i}a_{j,2}\det (M[X]_{i,j})  \right) +\left( \sum_{3\leq i,j\leq n}   a_{2,i}a_{j,1}\det (M[X]_{i,j})  \right). 
\end{multline}
By expanding down the first row and column of   $M[X\cup a]$ and of $M[X\cup b]$, we see the first and second bracketed terms on the right-hand side of \eqref{e.con2b} equal $\det(M[X\cup a])$ and $\det(M[X\cup b])$, respectively.  The remaining two sums in \eqref{e.con2b} are also determinants.
Let $N$ be the  $(n-1)\times(n-1)$ matrix whose first row is $  \begin{bmatrix} 0 & a_{1,3} & a_{1,4} & \cdots & a_{1,n}\end{bmatrix}$, first column is   $\begin{bmatrix} 0 & a_{3,2} & a_{4,2} & \cdots & a_{n,2}\end{bmatrix}^T$, with the rest of the matrix given by  $M[X]$.
Let $P$ be the $(n-1)\times(n-1)$ matrix whose first row is $  \begin{bmatrix} 0 & a_{2,3} & a_{2,4} & \cdots & a_{2,n}\end{bmatrix}$, first column is   $\begin{bmatrix} 0 & a_{3,1} & a_{4,1} & \cdots & a_{n,1}\end{bmatrix}^T$, with the rest of the matrix given by  $M[X]$.
By expanding down the first row and column of  the $N$ and $P$, we see the third and fourth bracketed terms on the right-hand side of \eqref{e.con2b} equal $\det(N)$ and $\det(P)$, respectively. However, since $M[X\cup \{a,b\}]$ is symmetric,  we see $N=P^T$, and so $\det(N)=\det(P)$. Since we are working over $GF(2)$, Equation~\eqref{e.con2a}, and so the theorem, holds.
\end{proof}

The following observation is an immediate consequence of Lemma~\ref{l.rep} and Theorem~\ref{t.hs}. It should be contrasted with the observations made in Remark~\ref{r.1}. We note that Corollary~\ref{c.3} is generalised by Theorem~\ref{th.cl}  where the assumption that the empty set is feasible is removed.
\begin{corollary}\label{c.3}
The set of binary delta-matroids in which the empty set is feasible is closed under handle slides. 
\end{corollary}

\begin{remark}
A proof of Theorem~\ref{t.hs} in the special case where $G$ is a  bouquet can be obtained from Theorem~\ref{con2}. The interlacement between, and the orientability of, edges of a bouquet $B$ can be used to obtain a matrix $M$ such that $D(B)=D(M)$ (see \cite{CMNR1} for a description of how).  By examining how  interlacement and orientability changes under a handle slide, it can be shown that $D(B_{ab})= D(M_{ab})$. Theorem~\ref{con2} then gives $D(M_{ab}) = D(M)_{ab}=D(B)_{ab}$. 
\end{remark}

Our starting point was the observation that handle slides can be used to put any bouquet into the form $B_{i,j,k}$. The following says that this result holds on the level of binary delta-matroids.
\begin{lemma}\label{t.3}
Let $D$ be a binary delta-matroid such that the empty set is feasible. Then there is a sequence of handle slides taking $D$ to $D_{i,j,k}$, for some $i,j,k$.
\end{lemma}
\begin{proof}
Since the empty set is feasible, by Lemma~\ref{l.rep}, $D=D(M)$ for some symmetric matrix $M$ over $GF(2)$.
We need to use handle slides and reordering of rows and columns to put $M$ in a block diagonal form in which each block is one of $\begin{bmatrix} 0\end{bmatrix}$, $\begin{bmatrix} 1 \end{bmatrix}$, or $\begin{bmatrix} 0 & 1 \\ 1 & 0\end{bmatrix}$. (It is clear that the delta-matroid of such a matrix equals $D_{i,j,k}$, for some $i,j,k$.)
To do this first observe that once we have a block of a matrix then a  handle slide $M_{ab}$ preserves that block as long as $a$ does not index a row or column of it. Thus, by induction, it is enough to show that we can always use handle slides to construct a block of the required form in the matrix $M$. 

If $M$ has a diagonal entry $m_{e,e}=1$. Then for each $f$ with  $m_{f,e}=m_{e,f}=1$  handle slide $f$ over $e$. In the resulting matrix,   all other entries of the $e$-th row and $e$-th column are zero, giving a block $\begin{bmatrix} 1 \end{bmatrix}$.

Now suppose all diagonal entries of $M$ are zero. If there is some $e$ such that all entries of the $e$-th row and $e$-th column are zero, then we have a block $\begin{bmatrix} 0\end{bmatrix}$.  
 Otherwise there is some $f$ with  $m_{f,e}=m_{e,f}=1$. For convenience, and without loss of generality, we can reorder the rows and columns so that $e$ labels the first row and column, and $f$ labels the second. So we have the submatrix  $\begin{bmatrix} 0 & 1 \\ 1 & 0\end{bmatrix}$ in the top left corner of $M$. We need to use handle slides to make all other entries in the first two rows and columns zero. This can be done as follows.
If $m_{i,e}=m_{e,i}=1$ and $m_{i,f}=m_{f,i}=0$ sliding $i$ over $f$ makes the $(i,e)$ and $(e,i)$ entries zero.
If $m_{i,e}=m_{e,i}=0$ and $m_{i,f}=m_{f,i}=1$ sliding $i$ over $e$ makes the $(i,f)$ and $(f,i)$ entries zero.
If $m_{i,e}=m_{e,i}=1$ and $m_{i,f}=m_{f,i}=1$ sliding $i$ over $f$, then $i$ over $e$  makes the four entries zero. Thus we can obtain a block  $\begin{bmatrix} 0 & 1 \\ 1 & 0\end{bmatrix}$, as required. This completes the proof of the lemma.
\end{proof}

We can now prove Theorem~\ref{t.1}.
\begin{proof}[Proof of Theorem~\ref{t.1}]
By Lemma~\ref{t.3}, there is a sequence of handle slides taking $D$ to $D_{i,j,k}$, for some $i,j,k$.  We have that $D_{i,j,k} = D(M_{i,j,k})$ where $M_{i,j,k}$ consist of $i$ blocks of $\begin{bmatrix} 0\end{bmatrix}$, $j$ blocks of  $\begin{bmatrix} 0 & 1 \\ 1 & 0\end{bmatrix}$, and $k$ blocks of the matrix $\begin{bmatrix} 1 \end{bmatrix}$. 

It is readily seen from Definition~\ref{d1} that handle slides of delta-matroids preserve parity, so $D$ is odd if and only if  $D_{i,j,k}$ is.   A delta-matroid $D(M)$, where $M$ is a symmetric matrix over $GF(2)$,  is odd if and only if there is a 1 on the diagonal of $M$ (this follows from the fact that a symmetric matrix of odd size over $GF(2)$ with zeros on the diagonal must be singular). Thus $D$ is even if and only if $D_{i,j,k}$ has $k= 0$, and the even case of the theorem follows.

Now suppose that $D$ is odd. Then handle slides can be used to put it in the form  $D_{i,j,k}$ with $k>0$. It remains to put this $D_{i,j,k}$ in the form $D_{i,0,p}$ for some $p\in \mathbb{N}$. 
If $j=0$ we are done, otherwise, possibly after reordering rows and columns, there is a block       $\begin{bmatrix} 0 & 1 &0\\ 1 & 0 &0\\ 0&0&1\end{bmatrix}$ whose rows and columns are labelled by $a,b,c$, say, in that order. The sequence of handle slides $a$ over $c$, $c$ over $b$, and $b$ over $a$ transforms this into the $3\times 3$ identity matrix. It follows that if $D_{i,j,k}$ has $k\neq 0$, then there is a sequence of handle slides taking $M$ to $D_{i,0,k+2j}$, completing the proof of the first part of the theorem.

For the second claim, suppose that there are sequences of handle slides take $D$ to $D_{i,j,k}$ and to $D_{p,q,r}$. Then there is a sequence of handle slides taking  $D_{i,j,k}$  to $D_{p,q,r}$.  Since a determinant of a block diagonal matrix is the product of the determinants of its blocks, the size of the largest feasible set in $D_{i,j,k}$ is $|E|-i$, and in  $D_{p,q,r}$ is $|E|-p$. Upon observing from Definition~\ref{d1}  that handle slides preserve the size of the largest feasible sets, we have that $i=p$, as required.
\end{proof}

It is worth emphasising that we have shown that if $D$ can be taken to $D_{i,j,k}$, then $2j+k$ is the size of the largest feasible set in $D$, and $i$ is the size of the ground set minus this number.

\begin{corollary}\label{c.2}
Let $D=(E,\mathcal{F})$ be a binary  delta-matroid in which the empty set is feasible and such that there is a sequence of handle slides taking $D$ to $D_{i,j,k}$.
\begin{enumerate}
\item Suppose $D$ is even. There is a sequence of handle slides taking $D$ to $D_{p,q,r}$ if and only if $p=i$, $q=j$, and $r=k=0$.
\item Suppose $D$ is odd.   There is a sequence of handle slides taking $D$ to $D_{p,q,r}$ if and only if $p=i$, $q=\ell$, and $r=|E|-i-2\ell$, for some $0\leq \ell\leq  \lfloor \frac{|E|-i}{2} \rfloor$.
\end{enumerate}
\end{corollary}
\begin{proof}
The first item follows from Theorem~\ref{t.1} upon noting that handle slides preserve parity. 
For the second item, suppose $D$ is odd. By Theorem~\ref{t.1}, $D_{i,\ell,|E|-i-2\ell}$ can be taken to $D_{i,0,|E|-i}$ using handle slides, and $D_{i,j,k}$ can be taken to $D_{i,0,|E|-i}$, thus $D$ can be taken to $D_{i,\ell,|E|-i-2\ell}$ by handle slides. Conversely, by Theorem~\ref{t.1}, $D_{i,j,k}$ and $D_{p,q,r}$ can both be taken to $D_{i,0,|E|-i}$ by handle slides, and so $i=p$ and $2q+r=|E|-i$, and result  follows.  
\end{proof}

Theorem~\ref{t.hs} can be used to show that ribbon graphic delta-matroids are not closed under handle slides. Choose a binary delta-matroid $D$ with empty set feasible that is not ribbon graphic. There is a sequence of handle slides taking $D$ to a ribbon graphic delta matroid $D_{i,j,k}$. Thus there must be a handle slide between a graphic and non-graphic delta-matroid. Despite this, the following result says that we can always work with handle slides within the class of ribbon graphic delta-matroids.   
\begin{corollary}\label{c.1}
Let $D=(E,\mathcal{F})$ be a ribbon graphic delta-matroid in which the empty set is feasible. If there is a sequence of handle slides taking $D$ to $D_{i,j,k}$, then there is a sequence of handle slides in which every delta-matroid is ribbon graphic that takes $D$ to $D_{i,j,k}$.
\end{corollary}
\begin{proof}
First suppose that $D$ is even. Then, by Theorem~\ref{t.1},  $D_{i,j,k}= D_{i,j,0}$. Since $D$ is ribbon graphic $D=D(B)$ for some bouquet $B$. By Proposition~\ref{hs} there is a sequence of (ribbon graph) handle slides taking $B$ to $B_{p,q,0}$. Taking the delta-matroids of the ribbon graphs that appear in this sequence and applying Theorem~\ref{t.hs}  gives a sequence of (delta-matroid) handle slides, in which every delta-matroid is ribbon graphic, that  takes $D$ to $D_{p,q,0}$. The result then follows by Corollary~\ref{c.2}.  

Now suppose that $D$ is odd. Then, by Corollary~\ref{c.2},  $ D_{i,j,k}=  D_{i,\ell,|E|-i-2\ell}$ for some $0\leq \ell\leq  \lfloor \frac{|E|-i}{2} \rfloor$.
 Since $D$ is ribbon graphic $D=D(B)$ for some bouquet $B$. By Proposition~\ref{hs} there is a sequence of (ribbon graph) handle slides taking $B$ to $B_{p,0,r}$. For a bouquet $H$ consisting of three non-interlaced non-orientable loops $a$, $b$, and $c$ whose ends appear in the order $aabbcc$ when travelling round the vertex, observe that $((H_{cb})_{ba})_{ac}$ consists of a pair of interlaced orientable loops $b$ and $c$, and a non-interlace non-orientable loop $a$. It follows that there is a   sequence of (ribbon graph) handle slides taking  $B_{p,0,r}$, and hence $B$, to $B_{p,m,r-2m}$ for each $0\leq m \leq  \lfloor \frac{|E|-p}{2} \rfloor$. Taking the delta-matroids of the ribbon graphs that appear in this sequence from $B$, and applying Theorem~\ref{t.hs},  gives a sequence of (delta-matroid) handle slides in which every delta-matroid is ribbon graphic that  takes $D$ to $D_{p,m,r-2m}$. By Corollary~\ref{c.2}, for some $m$, $D_{p,m,r-2m}=D_{i,\ell,|E|-i-2\ell}=D_{i,j,k}$, and the result follows.
\end{proof}

\begin{remark}
It is natural to ask if the binary condition in Theorem~\ref{t.1} can be dropped.  That is, can every delta-matroid in which the empty set is feasible be taken to a canonical form $D_{i,j,k}$ by a sequence of handle slides?
The answer is no. For example,  the delta-matroid over $E=\{1,2,3\}$ with feasible sets $\F=\{ \{1,2,3\}, \{1,2\}, \{1,3\}, \{2,3\},\emptyset\}$ cannot be. (Alternatively, that the answer is no follows  from Theorem~\ref{th.cl} below since the $D_{i,j,k}$ are binary.) However, there should be a version of Theorem~\ref{t.1} that includes non-binary delta-matroids or set systems. The key problem is determining the canonical forms (i.e., the analogues of the $D_{i,j,k}$, which may not be delta-matroids) for other classes of delta-matroids.
\end{remark}

\section{Closure under handle slides}\label{s.4}
Although handle slides are defined for all delta-matroids, because of our motivation from the classification of bouquets we have so far focused on delta-matroids in which the empty set is feasible. We now examine what happens when it is not. 

\begin{theorem}\label{th.cl}
The set of binary delta-matroids is closed under handle slides.
\end{theorem}
\begin{proof}
For any delta-matroid $D$, $A\subseteq E(D)$, and $a,b\in E(D)$ with $a\neq b$. If $a, b\notin A$,
\begin{equation}\label{e.cl1}
  D_{ab}\ast A=(D\ast A)_{ab}, 
\end{equation}
and if $a, b\in A$,
\begin{equation}\label{e.cl4}
 D_{ab}\ast A = (D\ast A)_{ba} . 
\end{equation}
Equation~\eqref{e.cl1} follows easily from the observation that, since $a, b\notin A$, for any $F\subseteq E(D)$, either of $a$ or $b$ is in $F$ if and only if it is in $F\btu A$. 
Equation~\eqref{e.cl4} follows by  direct computation. Start by writing 
\[  \F (D\ast A) = \left\{ X_i, Y_j\cup a, Z_k\cup b,  W_l\cup a, W_l\cup b , T_m\cup\{a,b\}\right\}_{i\in \mathcal{I},j\in \mathcal{J},k\in \mathcal{K},l\in \mathcal{L}, m\in \mathcal{M} }  , \]
where  $X_i, Y_j,Z_k,W_l \subseteq E\ba \{a,b\}$, none of the $Y_j,Z_k,W_l$ are equal, and where the $\mathcal{I}$,  $\mathcal{J}$, $\mathcal{K}$, $\mathcal{L}$, and $\mathcal{M}$ are indexing sets.
From this it is easy to compute the feasible sets of  $(D\ast A)_{ba}$, $D$, $D_{ab}$, and $D_{ab}\ast A$, upon which it is seen that  $\F (D\ast A) = \F( (D\ast A)_{ba})$, and Equation~\eqref{e.cl4} follows.

Now suppose that $D=(E,\F)$ is a binary delta-matroid and $a,b\in E$ with $a\neq b$. Then there is some $A\subseteq E$ and some symmetric  matrix $M$ over $GF(2)$ such that $D\ast A=D(M)$. We need to show that $D_{ab}$ is binary.  That is, we need to show that $D_{ab}\ast B=D(N)$ for some $B\subseteq E$ and some symmetric  matrix $N$ over $GF(2)$.
We will consider four cases given by the membership of $a$ and $b$ in $A$.

\noindent \underline{Case 1:} Suppose that $a,b\notin A$. Then, by Equation~\eqref{e.cl1} and Theorem~\ref{con2},
\begin{equation}\label{e.cl2}
  D_{ab}\ast A =    (D\ast A)_{ab} = D(M)_{ab}= D(M_{ab}),
\end{equation}
and so  $D_{ab}$ is binary.

\noindent \underline{Case 2:} Suppose that $a\in A$ and $b\in A$. Then, by Equation~\eqref{e.cl4} and Theorem~\ref{con2},
\begin{equation}\label{e.cl5}
  D_{ab}\ast A =    (D\ast A)_{ba} = D(M)_{ba}= D(M_{ba}),
\end{equation}
and so  $D_{ab}$ is binary.

\noindent \underline{Case 3:}  Suppose that $a\in A$ and $b\notin A$.
If there is some $F\in \F(D\ast A)$ with $a\in F$ and $b\notin F$, then by Lemma~\ref{l.rep}, we see that $D\ast(A\btu F)= D(N)$, for some   symmetric  matrix $N$ over $GF(2)$. Since $ a,b\notin A\btu F$, Case 1 applies and so $D_{ab}$ is binary.

Similarly, if there is some $F\in \F(D\ast A)$ with $a\notin F$ and $b\in F$, then by Lemma~\ref{l.rep}, $D\ast(A\btu F)= D(N)$, for some   symmetric  matrix $N$ over $GF(2)$. Since $ a,b\in A\btu F$, Case 2 now  applies and so $D_{ab}$ is binary.

Otherwise every feasible set of $D\ast A$ contains both $a$ and $b$, or neither of $a$ or $b$. Suppose this is the case. There is either some $F\in \F(D\ast A)$ containing both $a$ and $b$ or there is not.

First suppose that there is, and let $F\in \F(D\ast A)$ with   $a,b\in F$. Let  $X \in \F(D\ast A)$ be such that $a,b\notin X$ (we know such a set exists, since the empty set is feasible). 
Then $a\in X\btu F$, and by the Symmetric Exchange Axiom, $X\btu \{a,u\}\in \F$ for some $u\in  X\btu F$. Since, by hypothesis, $a$ or $b$ cannot appear  in a feasible set without the other, we must have $u=b$, and so   $X\cup \{a,b\}\in \F$. 
Similarly, the Symmetric Exchange Axiom gives that $F\btu \{a,u\}\in \F$ for some $u\in  X\btu F$. Again we must have that $b=u$ and so  $F\ba \{a,b\}\in \F$.
These two observations together give that we can partition the feasible sets of $D\ast A$ to get 
 $\F (D\ast A) = \left\{ X_i, X_i\cup\{a,b\}\right\}_{i\in \mathcal{I}} $, 
where  $X_i\subseteq E\ba \{a,b\}$ and   $\mathcal{I}$ is an indexing set.
From this we see that
 $\F (D) = \left\{ \hat{X}_i\cup a, \hat{X}_i\cup b \right\}_{i\in \mathcal{I}} $, 
where for each set $X_i$, $\hat{X_i}$ denotes $X_i\btu (A\ba\{a,b\})$, and that 
 $\F (D_{ab}) = \left\{ \hat{X}_i\cup b \right\}_{i\in \mathcal{I}}$. 
We then see that  $D_{ab}  = D\ba a$.  Since $D$ is binary, and the set of binary delta-matroids is minor-closed, it follows that $D_{ab} $ is binary.

All that remains is the case where no feasible set of $D\ast A$ contains $a$ or $b$ (so $a$ and $b$ are loops). In this case  each feasible set of $D$ contains $a$ but not $b$, and it follows that $D=D_{ab}$. Since $D$ is binary, so is $D_{ab}$.

\noindent \underline{Case 4:}  Suppose that $a\notin A$ and $b\in A$.
If there is some $F\in \F(D\ast A)$ with $a\notin F$ and $b\in F$, then, by Lemma~\ref{l.rep},  $D\ast(A\btu F)= D(N)$, for some   symmetric  matrix $N$ over $GF(2)$. Since $ a,b\notin  A\btu F$, Case 1 now  applies and so $D_{ab}$ is binary.

If there is some $F\in \F(D\ast A)$ with $a\in F$ and $b\notin F$, then  $D\ast(A\btu F)= D(N)$. Since $ a,b \in  A\btu F$, Case 2 now applies and so $D_{ab}$ is binary.

If there is some $F\in \F(D\ast A)$ with $a\in F$ and $b\in F$, then  $D\ast(A\btu F)= D(N)$. Since $ a\in  A\btu F$ and  $b\notin A\btu F$, Case 3 now applies and so $D_{ab}$ is binary.

All that remains is the case in which  no feasible set of $D\ast A$ contains $a$ or $b$ (so $a$ and $b$ are loops). In this case we can write
$\F (D\ast A) = \left\{ X_i\right\}_{i\in \mathcal{I}}  $, 
 where  $X_i\subseteq E\ba \{a,b\}$ and   $\mathcal{I}$ is an indexing set.
From this we see that
$\F (D_{ab} \ast A) = \left\{ \hat{X}_i\cup \{a, b\}, \hat{X}_i \right\}_{i\in \mathcal{I}} $, 
and that $D_{ab} \ast A = (D\ast A)  \oplus D_{0,1,0} $. Since both $D\ast A$ and $D_{0,1,0}$ are binary, it follows that $D_{ab}$ is. This completes the proof of the theorem. 
\end{proof}

Since handle slides preserve the maximum (and minimum) sizes of a feasible set in a delta-matroid, we have the following.
\begin{corollary}
The set of binary matroids is closed under slides.
\end{corollary}

The problem  of   extending Theorem~\ref{t.1} to all binary delta-matroids now arises.  
One way to try to extend the Theorem  is  to augment the set of canonical forms to include delta-matroids $D_{i,j,k,l}$ consisting of the direct sum of  $D_{i,j,k}$ with $l$ copies  of delta-matroids isomorphic to $(\{e\}, \{\{e\}\})$. We conjecture that a version of Theorem~\ref{t.1} holds for all binary delta-matroids with these terminal forms. 
\begin{conjecture}\label{conj}
For each binary delta-matroid $D$, there is a sequence of handle slides taking $D$ to some  $D_{i,j,k,l}$ where $i$ is the size of the ground set minus the size of a largest feasible set, $l$ is the size of a smallest feasible set, $2j+k$ is difference in the sizes of a largest and a smallest feasible set. Moreover, $k=0$ if and only if $D$ is even, and if $D$ is odd then every value of $j$ from 0 to $ \lfloor \frac{w}{2} \rfloor$, where $w$ is the difference between the sizes of a largest and a smallest feasible set, can be attained.
\end{conjecture}
Conjecture~\ref{conj} is true for ribbon graphic delta-matroids. This can be proven by induction on the size of a smallest feasible set. The base case is Theorem~\ref{t.1}. For the inductive step take a ribbon graph $G$ such that $D=D(G)$; choose any non-loop edge $e=(u,v)$ of $G$; handle slide each edge, other than $e$, that is incident with $u$ over $e$ so that $u$ becomes a degree 1 vertex; the corresponding handle slides in $D(G)$ transform it into a ribbon graphic delta-matroid with a direct summand $(\{e\}, \{\{e\}\})$. The main step in this argument is that any non-loop element $e$ of a ribbon graphic delta-matroid can be transformed into a coloop using only handle slides over $e$. This result does not hold for delta-matroids in general. For example, it is readily checked that  in the uniform matroid $U_{2,4}$ no element $e$ can be transformed into a coloop using only handle slides over $e$. In fact, with a little more work, it can be checked that no sequence of handle slides applied to $U_{2,4}$ will create a coloop. Of course the (delta-)matroid $U_{2,4}$ is not binary and so this example says nothing about the validity of Conjecture~\ref{conj}. However it does indicate that any approach to isolating $(\{e\}, \{\{e\}\})$ for the conjecture will be intimately tied to the binary structure of the delta-matroid.

It is perhaps also worth commenting on the alternative approach of considering sequences of twists, and handle slides that can only act on delta-matroids  in which the empty set is feasible, rather than just sequences of handle slides. Such an extension results in non-unique terminal forms.  For example if $D=(\{e,f\}, \{ \emptyset, \{e\}, \{e,f\} \})$ then there is a sequence of handle slides taking  $D$ to $D_{0,0,2}$, but also there is a sequence of handle slides taking $D\ast e$ to $D_{1,0,1}$. 
 In fact,  ribbon graph theory indicates that this approach should fail. The ribbon graph  analogue is to consider ribbon graphs up to partial duals (see \cite{Ch09,CMNR1}), and handle slides that can only act on bouquets. But  partial duality changes the topology of a surface, and so our choice of terminal form will need to reflect this. Nevertheless, this relation on binary delta-matroids will result in some set of terminal forms. What are they? 

We conclude with one final open question. We have seen that binary delta-matroids are closed under handle slides, but that delta-matroids, in general, are not. What classes of delta-matroids are closed under handle slides?

\section*{Acknowledgements}
This work arose from a visit supported by a Scheme 5 - International Short Visits grant from the London Mathematical Society. We would like to thank the LMS for their support. We would also like the thank Robert Brijder for his  helpful comments and suggestions.

  \bibliographystyle{plain} 
    \bibliography{bib}

\end{document}